\documentclass[letterpaper, 10 pt, conference]{ieeeconf}  

\IEEEoverridecommandlockouts                              

\usepackage{graphics} 
\usepackage{subfig}
\usepackage{epsfig} 
\usepackage{amsmath} 
\usepackage{amssymb}  
\usepackage{mathrsfs}
\usepackage{cite}
\usepackage{epstopdf}
\usepackage{hyperref}

\newtheorem{theorem}{\textbf{Theorem}}

\newtheorem{corollary}{\textbf{Corollary}}
\newtheorem{assumption}{\textbf{Assumption}}
\newtheorem{definition}{\textbf{Definition}}
\newtheorem{example}{\textbf{Example}}
\newtheorem{remark}{\textbf{Remark}}

\newcommand{\Real}{\mathbb R}

\graphicspath{{Images/}}

\title{\LARGE \bf Distributed Feedback Controllers for Stable Cooperative Locomotion of Quadrupedal Robots: A Virtual Constraint Approach*}

\author{Kaveh Akbari Hamed$^{1}$, Vinay R. Kamidi$^{1}$, Abhishek Pandala$^{1}$, Wen-Loong Ma$^{2}$, and  Aaron D. Ames$^{2}$
\thanks{*The work of K. Akbari Hamed is supported by the National Science Foundation (NSF) under Grant Numbers 1854898, 1906727, 1923216, and 1924617. The work of V. R. Kamidi and A. Pandala is supported by the NSF under the Grant Number 1854898. The work of A. D. Ames is supported by the NSF under Grant Numbers 1544332, 1724457, 1724464, 1923239, and 1924526 as well as Disney Research LA. The content is solely the responsibility of the authors and does not necessarily represent the official views of the NSF.}
\thanks{$^{1}$K. Akbari Hamed, V.R. Kamidi, and A. Pandala are with the Department of Mechanical Engineering, Virginia Tech, Blacksburg, VA 24061, USA, {\tt\small kavehakbarihamed@vt.edu}, {\tt\small vinay28@vt.edu}, and {\tt\small agp19@vt.edu}}
\thanks{$^{2}$W. Ma and A. D. Ames are with the Department of Mechanical and Civil Engineering, California Institute of Technology, Pasadena, CA 91125, USA, {\texttt{\small{wma@caltech.edu}} and \texttt{\small{ames@cds.caltech.edu}}}}
}

\begin{document}

\maketitle
\thispagestyle{empty}
\pagestyle{empty}


\begin{abstract}
This paper aims to develop distributed feedback control algorithms that allow cooperative locomotion of quadrupedal robots which are coupled to each other by holonomic constraints. These constraints can arise from collaborative manipulation of objects during locomotion. In addressing this problem, the complex hybrid dynamical models that describe collaborative legged locomotion are studied. The complex periodic orbits (i.e., gaits) of these sophisticated and high-dimensional hybrid systems are investigated. We consider a set of virtual constraints that stabilizes locomotion of a single agent. The paper then generates modified and local virtual constraints for each agent that allow stable collaborative locomotion. Optimal distributed feedback controllers, based on nonlinear control and quadratic programming, are developed to impose the local virtual constraints. To demonstrate the power of the analytical foundation, an extensive numerical simulation for cooperative locomotion of two quadrupedal robots with robotic manipulators is presented. The numerical complex hybrid model has 64 continuous-time domains, 192 discrete-time transitions, 96 state variables, and 36 control inputs.
\end{abstract}


\section{INTRODUCTION}
\label{INTRODUCTION}
\vspace{-0.5em}

This paper aims to develop a formal foundation, based on hybrid systems theory, nonlinear control, and quadratic programming (QP), to develop distributed feedback control algorithms that stabilize cooperative locomotion of quadrupedal robots while steering objects. Legged robots that are augmented with manipulators can form \textit{collaborative robot (co-robot) teams} that assist humans in different aspects of their life such as labor-intensive tasks, construction, and manufacturing. Although important theoretical and technological advances have allowed the development of distributed controllers for motion control of complex robot systems, state-of-the-art approaches address the control of multiagent systems composed of collaborative robotic arms \cite{Murray_Control_primitives}, multifingered robot hands \cite{Murray_Book}, aerial vehicles \cite{RSS2013,MT:NM:14}, and ground vehicles \cite{DP:MT:14,Mesbahi_Book,Bullo_Book}, but \textit{not} cooperative legged agents. Legged robots are \textit{inherently unstable}, as opposed to most of the systems where these algorithms have been deployed. Furthermore, the evolution of legged co-robot teams that cooperatively manipulate objects can be represented by high-dimensional and \textit{complex hybrid dynamical systems} which complicate the design of distributed control algorithms.

\begin{figure}[t!]
\centering
\includegraphics[width=0.9\linewidth]{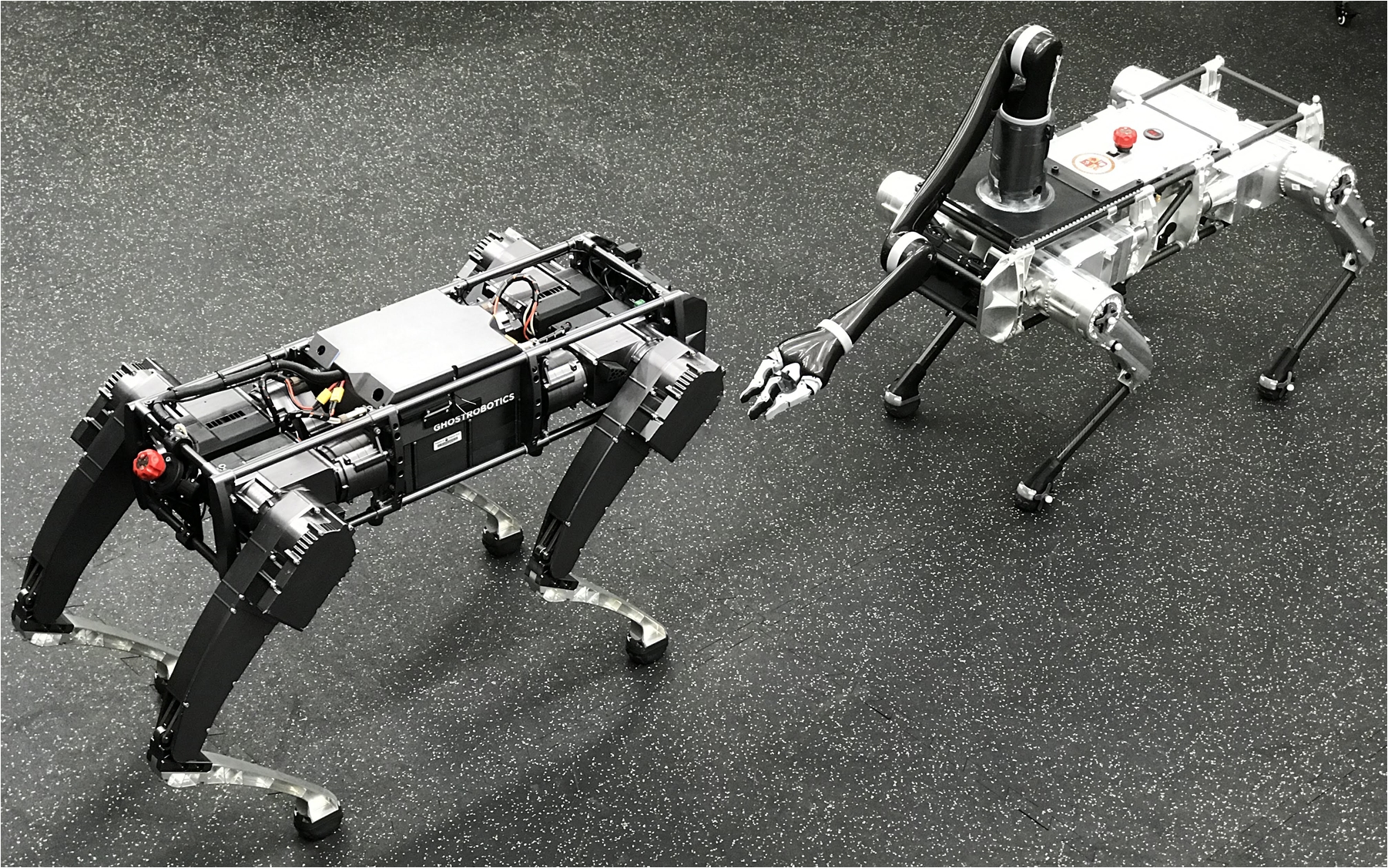}
\vspace{-0.3em}
\caption{Two Vision 60 robots, augmented with a Kionva arm, whose full-order models will be used for the numerical simulations of cooperative locomotion.}
\label{Two_V60s}
\vspace{-2em}
\end{figure}

\begin{figure*}[t!]
\centering
\vspace{0.07em}
\includegraphics[width=\linewidth]{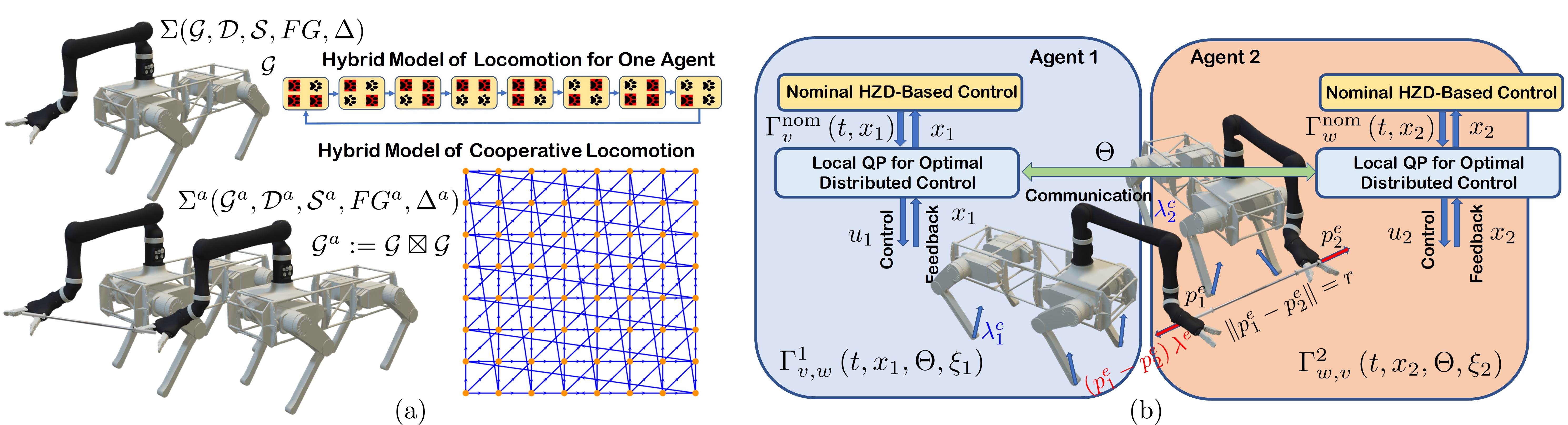}
\vspace{-1.5em}
\caption{(a) Illustration of the hybrid models for locomotion of one agent and two agents. The figure depicts the digraphs for $8$-domain walking locomotion as well as $64$-domain cooperative locomotion. (b) Illustration of the proposed distributed feedback control algorithms based on HZD and local QPs.}
\label{Unified_Illsuartion}
\vspace{-2em}
\end{figure*}

\vspace{-0.5em}
\subsection{Related Work and Motivation}
\vspace{-0.1em}

Hybrid systems theory has become a powerful tool to design nonlinear controllers for dynamic legged locomotion both in theory and practice \cite{Grizzle_Asymptotically_Stable_Walking_IEEE_TAC,Westervelt_Grizzle_Koditschek_HZD_IEEE_TRO,Chevallereau_Grizzle_3D_Biped_IEEE_TRO,Ames_RES_CLF_IEEE_TAC,Ames_DURUS_TRO,Sreenath_Grizzle_HZD_Walking_IJRR,Park_Grizzle_Finite_State_Machine_IEEE_TRO,Poulakakis_Grizzle_SLIP_IEEE_TAC,Tedrake_Robus_Limit_Cycles_CDC,Byl_HZD,Johnson_Burden_Koditschek,Spong_Controlled_Symmetries_IEEE_TAC,Manchester_Tedrake_LQR_IJRR,Ioannis_adaptation_01,Vasudevan2017,Hamed_Buss_Grizzle_BMI_IJRR}. Nonlinear controllers that address hybrid nature of legged locomotion has come out of hybrid reduction \cite{Ames_HybridReduction_Original_Paper}, controlled symmetries \cite{Spong_Controlled_Symmetries_IEEE_TAC}, transverse linearization \cite{Manchester_Tedrake_LQR_IJRR}, and hybrid zero dynamics (HZD) \cite{Westervelt_Grizzle_Koditschek_HZD_IEEE_TRO,Ames_RES_CLF_IEEE_TAC}. HZD-based controllers have been numerically and experimentally evaluated for bipedal locomotion \cite{Sreenath_Grizzle_HZD_Walking_IJRR,Jessy_Book,Ramezani_Hurst_Hamed_Grizzle_ATRIAS_ASME,Buss_Hamed_Griffin_Grizzle_BMI_ACC,Ames_RES_CLF_IEEE_TAC,Ames_DURUS_TRO,Manchester_Tedrake_LQR_IJRR,Martin_Schmiedeler_IJRR}, quadrupedal locomotion \cite{Hamed_Ma_Ames_Vision60,Ioannis2016}, powered prosthetic legs \cite{Gregg_Virtual_Constraints_Powered_Prosthetic_IEEE_TRO}, and exoskeletons \cite{Grizzle_Decentralized}.

Existing nonlinear control approaches for legged robots are tailored to the path planning and stabilization of dynamic gaits for one legged machine, but \textit{not} complex hybrid dynamical models that describe the evolution of multiagent legged robotic systems. This is mainly due to the fact that state-of-the-art techniques for dynamic locomotion are centralized approaches that \textit{cannot} be easily transferred to legged co-robot teams. A legged co-robot team that manipulates an object can be modeled as a set of legged robots which are coupled to each other and the object by a set of holonomic constraints. The question is how to construct controllers for such a complex robotic system that control locomotion with many DOFs and large amounts of sensory data? Computing the control torques for such a composite system in 1kHz is often impossible---\textit{this motivates the importance of developing distributed and decentralized controllers that address lower-dimensional subsystems (e.g., each agent)}.

\subsection{Objectives and Contributions}

The \textit{objectives} and \textit{contributions} of this paper are as follows. We address complex hybrid models that describe cooperative locomotion of legged robotic systems. The properties as well as periodic solutions (i.e., complex gaits) of these sophisticated hybrid dynamical models are investigated. We develop a distributed feedback control algorithm, based on HZD and distributed QPs, to stabilize cooperative gaits. We study virtual constraints and nominal HZD-based controllers that stablize locomotion of one agent. We then modify the virtual constraint controllers for the cooperative locomotion of two agents. A QP formulation is set up to compute the optimal  distributed controllers that are close to the nominal HZD controllers of each agent while imposing the modified virtual constraints. To demonstrate the power of the analytical foundation, an extensive numerical simulation of two quadrupedal agents (Vision 60 robots) with Kinova arms steering an object is presented (see Fig. \ref{Two_V60s}). In this simulation, the complex hybrid model has 64 continuous-time domains, 192 discrete transitions, 96 state variables, and 36 control inputs. Our previous work in \cite{Hamed_Gregg_IEEE_TAC,Hamed_Gregg_decentralized_control_IEEE_CST} developed an optimization algorithm, based on Poincar\'e return maps and linear and bilinear matrix inequalities (LMIs and BMIs) to synthesize decentralized controllers for legged locomotion. However, the Jacobian linearization of the Poincar\'e map for the above-mentioned complex hybrid model with 64 continuous-time models and 96 state variables is computationally intensive. The current paper addresses this complexity with the proper modification of local virtual constraints and the QP formulation. The current work is different from \cite{Hamed_Kamidi_Ma_Leonessa_Ames_Dog_Human} which addresses cooperative locomotion of guide robots and humans. In particular, \cite{Hamed_Kamidi_Ma_Leonessa_Ames_Dog_Human} considers a leash structure with a ``variable'' and ``controlled'' length and angle that stabilize the cooperative locomotion of quadrupedal robots and humans. In the current study, the object is ``passive'' with a ``fixed length''. The paper instead develops distributed controllers for each agent that stabilize the cooperative locomotion subject to \textit{holonomic constraints}. The current work is also different from the study presented in \cite{Ioannis_adaptation_01} for locomotion adaptation of limit cycle bipedal walkers in leader/follower collaborative tasks. The current paper addresses complex models of two agents while designing local and optimal distributed controllers for stable collaborative locomotion. Reference \cite{Ioannis_adaptation_01} only considers the follower dynamics while  designing a switching based controller for its adaptation to a persistent external force that represents the leader.


\vspace{-0.4em}
\section{MODELS OF COOPERATIVE LOCOMOTION}
\label{HYBRID MODELS OF COOPERATIVE LOCOMOTION}

\subsection{Hybrid Model of Locomotion for One Agent}
\label{Hybrid Model of Locomotion for One Agent}

We consider a full-order dynamical model of Vision 60 that is augmented with a Kinova arm for the locomotion and manipulation purposes (see Fig. \ref{Two_V60s}). Vision 60 is an autonomous quadrupedal robot that is designed and manufactured by Ghost Robotics\footnote{\url{https://www.ghostrobotics.io/}}.
Vision 60 has 18 DOFs of which 12 leg DOFs are actuated. In particular, each leg of the robot consists of a 1 DOF actuated knee joint with pitch motion and a 2 DOF actuated hip joint with pitch and roll motions. The remaining 6 DOFs are associated with the translational and rotational movements of the body. In this paper, we consider a 6 DOF Kinova arm that is affixed on Vision 60. The generalized coordinates vector of each agent is represented by $q:=\textrm{col}(p_{b},\phi_{b},q_{\textrm{body}})\in\mathcal{Q}\subset\Real^{24}$, where $p_{b}\in\Real^{3}$ and $\phi_{b}\in\Real^{3}$ denote the absolute position and orientation (i.e., roll-pitch-yaw) of the robot with respect to a world frame, respectively. In addition,  $q_{\textrm{body}}$ denotes a set of coordinates that describe the shape of Vision 60 together with the arm, and $\mathcal{Q}$ is the configuration space. The state and control inputs (i.e., joint torques) of the system are denoted by $x:=\textrm{col}(q,\dot{q})\in\mathcal{X}:=\textrm{T}\mathcal{Q}\subset\Real^{48}$ and $u\in\mathcal{U}\subset\Real^{18}$, respectively. In our notation, $\mathcal{X}$ and $\mathcal{U}$ represent the state manifold and set of admissible controls.

Throughout this paper, we will consider a hybrid systems formulation to describe \textit{multi-domain} quadrupedal locomotion for each agent as  $\Sigma\left(\mathcal{G},\mathcal{D},\mathcal{S},FG,\Delta\right)$, in which $\mathcal{G}=(\mathcal{V},\mathcal{E})$ denotes the \textit{directed cycle} corresponding to the desired locomotion pattern with the \textit{vertices set} $\mathcal{V}$ and the \textit{edges set} $\mathcal{E}\subseteq\mathcal{V}\times\mathcal{V}$. In our formulation, the vertices represent the \textit{continuous-time domains} of quadrupedal locomotion that are described by ordinary differential equations (ODEs) arising from the Lagrangian dynamics. The edges represent the \textit{discrete-time transitioning} between continuous-time domains arising from the changes in physical constraints. The evolution of the mechanical system during the domain $v\in\mathcal{V}$ is described by the ODE $\dot{x}=f_{v}(x)+g_{v}(x)\,u$ for all $(x,u)\in\mathcal{D}$, where $\mathcal{D}$ represents the domains of admissibility.  The set of control systems is given by $FG:=\{(f_{v},g_{v})\}_{v\in\mathcal{V}}$. The evolution of the system during the discrete-time transition $e\in\mathcal{E}$ is then described by the reset law (i.e., reinitialization rule) $x^{+}=\Delta_{e}(x^{-})$, where $x^{-}$ and $x^{+}$ represent the state of the system right before and after the transition, respectively. Moreover, $\Delta:=\{\Delta_{e}\}_{e\in\mathcal{E}}$ is the set of discrete-time dynamics. Finally, the guards of the hybrid system $\Sigma$ is given by the switching manifolds $\mathcal{S}:=\{\mathcal{S}_{e}\}_{e\in\mathcal{E}}$ on which the discrete-time transitions occur.


\textit{Continuous-Time Dynamics:} Let us assume that $J_{v}^{c}(q)$ denotes the contact Jacobian matrix during the continuous-time domain $v\in\mathcal{V}$, where the superscript ``$c$'' stands for the contact. Then, the evolution of the mechanical system in domain $v$ can be described by the following ODE:
\begin{alignat}{4}
&D(q)\,\ddot{q}&&+H(q,\dot{q})&&=B\,u+J_{v}^{c\top}(q)\,\lambda^{c}\label{dyn_02}\\
&J_{v}^{c}(q)\,\ddot{q}&&+\dot{J}_{v}^{c}(q,\dot{q})&&=0,\label{dyn_03}
\end{alignat}
where $D(q)$ is the mass-inertia matrix, $H(q,\dot{q})$ represents the Coriolis, centrifugal, and gravitational terms, $B$ is the input distribution matrix, $\lambda^{c}$ denotes the ground reaction forces (i.e., Lagrange multipliers), and $\dot{J}_{v}^{c}(q,\dot{q}):=\frac{\partial}{\partial q}(J_{v}^{c}(q)\,\dot{q})\,\dot{q}$. By eliminating $\lambda^{c}$, one can express \eqref{dyn_02} and \eqref{dyn_03} as the input-affine system $\dot{x}=f_{v}(x)+g_{v}(x)\,u$.

\textit{Discrete-Time Dynamics:} The \textit{next-domain function} $\mu:\mathcal{V}\rightarrow\mathcal{V}$ is defined as $\mu(v)=w$ if the vertices $v$ and $w$ are adjacent on $\mathcal{G}$, or equivalently, if $e=(v\rightarrow{}w)\in\mathcal{E}$. If during the discrete transition $e$, an existing contact breaks, we define the reset law $\Delta_{e}(x)$ as identity to preserve the continuity of position and velocity. However, if a new contact point is added to the existing set of contacts points with the ground, we make use of the rigid impact model as follows \cite{Hurmuzlu_Impact}
\begin{alignat}{4}
D(q)\left(\dot{q}^{+}-\dot{q}^{-}\right)=J_{\mu(v)}^{c\top}\,\delta\lambda^{c},\quad J_{\mu(v)}^{c}(q)\,\dot{q}^{+}=0,\label{dyn_05}
\end{alignat}
to model the abrupt changes in the velocity components according to the impact. Here, $\delta\lambda^{c}$ represents the intensity of the impulsive ground reaction forces. From \eqref{dyn_05} and continuity of position, one can obtain the impact map as $x^{+}=\Delta_{e}(x^{-})$.


\vspace{-0.2em}
\subsection{Continuous-Time Models for Cooperative Locomotion}
\label{Continuous-Time Models for Cooperative Locomotion}

In this section, we will address models of continuous-time domains that describe cooperative locomotion of two agents while steering an object. In our notation, the subscript $i\in\{1,2\}$ represents the agent number. The state variables and control inputs for the agent $i$ are represented by $x_{i}:=\textrm{col}(q_{i},\dot{q}_{i})\in\mathcal{X}$ and $u_{i}\in\mathcal{U}$, respectively.
To simplify the analysis, we will assume that the agents are identical. Let $p^{e}(q_{i})\in\Real^{3}$ denote the Cartesian coordinates of the end effector (EE) with respect to the world frame. We consider a holonomic constraint between the EEs as follows (see Fig. \ref{Unified_Illsuartion}a):
\begin{equation}
    \left\|p^{e}(q_{1})-p^{e}(q_{2})\right\|_{2}^{2}
    =\textrm{constant}.\label{holonomic_EE}
\end{equation}
Equation \eqref{holonomic_EE} states that the Euclidean distance between the agents' EEs is constant. In particular, we assume that the agents are carrying a massless bar with a fixed length whose ends are connected to the agents' EEs via socket (i.e., ball) joints. Differentiating \eqref{holonomic_EE} results in $(p_{1}^{e}-p_{2}^{e})^\top(J^{e}_{1}\,\dot{q}_{1} - J_{2}^{e}\,\dot{q}_{2})=0$ and $(p_{1}^{e}-p_{2}^{e})^\top(J^{e}_{1}\,\ddot{q}_{1} - J_{2}^{e}\,\ddot{q}_{2})
+\|\dot{p}_{1}^{e} - \dot{p}_{2}^{e}\|^{2}=0$, where $J^{e}(q):=\frac{\partial p^{e}}{\partial q}(q)$, $p_{i}^{e}:=p^{e}(q_{i})$, $J_{i}^{e}:=J^{e}(q_{i})$, and $\dot{p}_{i}^{e}=J_{i}^{e}\,\dot{q}_{i}$ for $i\in\{1,2\}$. Let us assume that $\lambda^{e}\in\Real$ denotes the Lagrange multiplier corresponding to the holonomic constraint \eqref{holonomic_EE}. Suppose further that the agents 1 and 2 are in the continuous-time domains $v\in\mathcal{V}$ and $w\in\mathcal{V}$, respectively. Then, the equations of motions can be described by the following coupled and constrained dynamics:
\begin{alignat}{6}
&D_{1}\,\ddot{q}_{1} \,\,\,+ H_{1} \,\,\,\,= B\,u_{1} + J_{v,1}^{c\top}\,\lambda^{c}_{1} \,+ J_{1}^{e\top} \left(p_{1}^{e} - p_{2}^{e}\right) \lambda^{e}\label{dyn_06}\\
&D_{2}\,\ddot{q}_{2} \,\,\,+ H_{2} \,\,\,\,= B\,u_{2} + J_{w,2}^{c\top}\,\lambda^{c}_{2} + J_{2}^{e\top} \left(p_{2}^{e} - p_{1}^{e}\right) \lambda^{e}\label{dyn_07}\\
&J_{v,1}^{c}\,\ddot{q}_{1} \,\,+ \dot{J}_{v,1}^{c} \,\,= 0\label{dyn_08}\\
&J_{w,2}^{c}\,\ddot{q}_{1} \,+ \dot{J}_{w,2}^{c} \,= 0\label{dyn_09}\\
&\left(p_{1}^{e}-p_{2}^{e}\right)^\top\left(J^{e}_{1}\,\ddot{q}_{1} - J_{2}^{e}\,\ddot{q}_{2}\right)+\left\|\dot{p}_{1}^{e} -\dot{p}_{2}^{e}\right\|^{2} = 0,\label{dyn_10}
\end{alignat}
where $D_{i}:=D(q_{i})$, $H_{i}:=H(q_{i},\dot{q}_{i})$, $J_{v,i}^{c}:=J_{v}^{c}(q_{i})$, and $\dot{J}_{v,i}^{c}:=\dot{J}_{v}^{c}(q_{i},\dot{q}_{i})$ for $i\in\{1,2\}$. We remark that in \eqref{dyn_06} and \eqref{dyn_07}, $(p_{1}^{e}-p_{2}^{e})\,\lambda^{e}\in\Real^{3}$ and $(p_{2}^{e}-p_{1}^{e})\,\lambda^{e}\in\Real^{3}$ represent the forces associated with the holonomic constraint \eqref{holonomic_EE} that are aligned with the bar and applied to the EEs of the agents 1 and 2, respectively. By eliminating the Lagrange multipliers $\lambda_{1}^{c}$, $\lambda_{2}^{c}$, and $\lambda^{e}$ from \eqref{dyn_06}-\eqref{dyn_10}, one can express the evolution of the composite mechanical system by $\dot{x}^{a} = f_{v,w}^{a}\left(x^{a}\right) + g_{v,w}^{a}\left(x^{a}\right) u^{a}$,
where the superscript ``$a$'' stands for the augmented system, and $x^{a}:=\textrm{col}(x_{1},x_{2})\in\mathcal{X}\times\mathcal{X}$ and $u^{a}:=\textrm{col}(u_{1},u_{2})\in\mathcal{U}\times\mathcal{U}$ denote the augmented state and control inputs, respectively.


\vspace{-0.2em}
\subsection{Discrete-Time Models for Cooperative Locomotion}
\label{Discrete-Time Models for Cooperative Locomotion}

This section addresses the discrete-time transition for the composite mechanical system. We consider a general case in which the agents 1 and 2 can switch from any continuous-time domains $v$ and $w$ to $v'$ and $w'$, respectively. This discrete-time transition is denoted by $(v,w)\rightarrow(v',w')$. We define the extended contact Jacobian matrix for the agent 1 as $\hat{J}_{v\rightarrow{v'}}^{c}(q_{1})$ by $\hat{J}_{v\rightarrow{v'}}^{c}(q_{1}):=J_{v}^{c}(q_{1})$ if $v'=v$ and $\hat{J}_{v\rightarrow{v'}}^{c}(q_{1}):=J_{\mu(v)}^{c}(q_{1})$ for $v'\neq{v}$. An analogous extended contact Jacobian matrix $\hat{J}_{w\rightarrow{w'}}^{c}(q_{2})$ can be defined for the agent 2. The evolution of the composite mechanical system over the infinitesimal period of the impact can be then described by the following coupled dynamics
\begin{alignat}{4}
&D_{1}\left(\dot{q}_{1}^{+} - \dot{q}_{1}^{-}\right)= \hat{J}_{v\rightarrow{}v',1}^{c\top}\,\delta\lambda_{1}^{c}\,+J_{1}^{e\top}\left(p_{1}^{e}-p_{2}^{e}\right)\delta\lambda^{e}\label{dyn_11}\\
&D_{2}\left(\dot{q}_{2}^{+} - \dot{q}_{2}^{-}\right)= \hat{J}_{w\rightarrow{}w',2}^{c\top}\,\delta\lambda_{2}^{c}+J_{2}^{e\top}\left(p_{2}^{e}-p_{1}^{e}\right)\delta\lambda^{e}\label{dyn_12}\\
&\hat{J}_{v\rightarrow{}v',1}^{c}\,\dot{q}_{1}^{+}=0\label{dyn_13}\\
&\hat{J}_{w\rightarrow{}w',2}^{c}\,\dot{q}_{2}^{+}=0\label{dyn_14}\\
&\left(p_{1}^{e}-p_{2}^{e}\right)^\top\left(J^{e}_{1}\,\dot{q}_{1}^{+} - J_{2}^{e}\,\dot{q}_{2}^{+}\right)=0,\label{dyn_15}
\end{alignat}
where $\delta\lambda^{c}_{1}$, $\delta\lambda^{c}_{2}$, and $\delta\lambda^{e}$ represent the intensity of the impulsive Lagrange multipliers at the leg ends and EEs. By eliminating the Lagrange multipliers from \eqref{dyn_11}-\eqref{dyn_15}, one can obtain an augmented reset law as $x^{a+}=\Delta_{(v,w)\rightarrow(v',w')}^{a}(x^{a-})$.


\subsection{Complex Hybrid Model for Cooperative Locomotion}
\label{Complex Hybrid Model for Cooperative Locomotion}

The complex hybrid model that describes the cooperative locomotion of two robots will have a complex graph that is taken as the strong product of graph $\mathcal{G}=(\mathcal{V},\mathcal{E})$ with itself. In particular, this strong product is represented by $\mathcal{G}^{a}:=\mathcal{G}\boxtimes\mathcal{G}=(\mathcal{V}^{a},\mathcal{E}^{a})$ that has the vertices set $\mathcal{V}^{a}:=\mathcal{V}\times\mathcal{V}$. Furthermore to define the edges set $\mathcal{E}^{a}$, we remark that any $e=\left((v,w)\rightarrow(v',w')\right)\in\mathcal{E}^{a}$ if and only if 1) $v=v'$ and $(w\rightarrow{}w')$ is an edge in $\mathcal{E}$, or 2) $(v\rightarrow{}v')$ is an edge in $\mathcal{E}$ and $w=w'$, or 3) $(v\rightarrow{}v')$ is an edge in $\mathcal{E}$ and $(w\rightarrow{}w')$ is an edge in $\mathcal{E}$. Finally, the complex hybrid model is defined as
$\Sigma^{a}\left(\mathcal{G}^{a},\mathcal{D}^{a},\mathcal{S}^{a},FG^{a},\Delta^{a}\right)$,
in which $FG^{a}:=\{(f_{v,w}^{a},g_{v,w}^{a})\}_{(v,w)\in\mathcal{V}^{a}}$, $\Delta^{a}:=\{\Delta^{a}_{e}\}_{e\in\mathcal{E}^{a}}$, and $\mathcal{D}^{a}$ and $\mathcal{S}^{a}$ denote the augmented sets of admissibility and switching surfaces, respectively.

\begin{example}
In this paper, we will consider walking gait of Vision 60 whose graph $\mathcal{G}$ is assumed to have $8$ continuous-time domains and $8$ discrete-time domains. It can be shown that $\mathcal{G}^{a}:=\mathcal{G}\boxtimes\mathcal{G}$ has $64$ continuous-time domains and $192$ discrete-time domains (see \ref{Unified_Illsuartion}a).
\end{example}


\vspace{-0.4em}
\section{DISTRIBUTED FEEDBACK CONTROLLERS}
\label{DISTRIBUTED FEEDBACK CONTROLLERS}

We will consider complex periodic locomotion of the composite robot. This section aims to present the structure of the proposed distributed feedback control algorithms that will stabilize complex dynamic gaits for cooperative locomotion of two legged robots. We will also investigate some properties of the complex hybrid model. Sections \ref{HZD-BASED NOMINAL CONTROLLERS} and \ref{SYNTHESIS OF DISTRIBUTED CONTROLLERS BASED ON LOCAL QPS} will synthesize distributed controllers for the cooperative locomotion. We define a periodic gait for one agent as follows.

\begin{assumption}\label{Periodic Gait of One Agent}
\textit{\textbf{(Periodic Gait of One Agent):}} We suppose that there is a family of \textit{nominal} state feedback laws $\Gamma^{\textrm{nom}}(t,x):=\{\Gamma_{v}^{\textrm{nom}}(t,x)\}_{v\in\mathcal{V}}$ that generates a periodic orbit (i.e., gait) for the hybrid model of one agent $\Sigma$. In particular, there is a periodic state trajectory (i.e., solution) $\varphi^{\star}:[0,\infty)\rightarrow\mathcal{X}$ for $\Sigma$ if we apply the nominal state feedback law $u=\Gamma_{v}^{\textrm{nom}}(t,x)$ during the continuous-time domain $v$ for all $v\in\mathcal{V}$. The corresponding \textit{periodic orbit} is given by $\mathcal{O}:=\{x=\varphi^{\star}(t)\,|\,0\leq{t}<T\}$ for some fundamental period $T>0$.
\end{assumption}

We then present the concept of measurable global variables for both agents (i.e., subsystems) as follows.

\begin{assumption}\label{Measurable Global Variables}
\textit{\textbf{(Measurable Global Variables):}} There is a set of quantities $\Theta$ that (i) depend on the global state variables $x^{a}=\textrm{col}(x_{1},x_{2})$, i.e., $\Theta=\Theta(x^{a})$ and (ii) are measurable for both subsystems via sensors. These variables are referred to as the \textit{measurable global variables}. We further assume that agents know the domain numbers of each other. In particular, $(v,w)\in\mathcal{V}^{a}$ is known for both agents.
\end{assumption}

Using this hypothesis, we propose a \textit{parameterized family of local controllers} for the agent $i\in\{1,2\}$ as follows
\begin{equation}\label{distributed_controllers}
u_{i}=\Gamma^{i}_{v,w}\left(t,x_{i},\Theta,\xi_{i}\right), \quad (v,w)\in\mathcal{V}^{a}
\end{equation}
(see Fig. \ref{Unified_Illsuartion}b). Here, $\Gamma^{i}$ is a local state feedback law that has access to 1) the local state variables of the agent $i$, i.e., $x_{i}$, 2) the measurable global variables $\Theta(x^{a})$, and 3) the domain number of the other agent that is denoted by $w$. Furthermore, this local feedback law is parameterized by a set of adjustable local controller parameters $\xi_{i}$ to achieve stability of the complex gait. More specifically, we will show that the stability of the cooperative gaits will depend on the proper selection of local controller parameters $\xi_{i}$, $i\in\{1,2\}$.

Before we present the complex periodic gaits for cooperative locomotion of two agents, we consider the following assumption on distributed feedback controllers.

\begin{assumption}\label{Invariance to x and y}
The family of local feedback laws $\Gamma^{i}_{v,w}(t,x_{i},\Theta,\xi_{i})$ for $i\in\{1,2\}$ do \textit{not} depend on the horizontal displacements (i.e., Cartesian coordinates) of the robots on the walking surface. However, they are allowed to depend on the translational velocities of the robots.
\end{assumption}

\begin{definition}
\textit{\textbf{(Translation Operator):}} We define the \textit{translation operator} on $\mathcal{X}$ by $T_{d}(x)=T_{d}(q,\dot{q})$ which takes the state vector $x$ and adds the vector $d\in\Real^{2}$ to its Cartesian positions in the horizontal plane. This corresponds to moving the robot on the walking surface by the vector $d$ while keeping the other state variables of the robot unchanged.
\end{definition}

Now we are in a position to present the following result for complex gaits of cooperative locomotion.

\begin{theorem}\label{Complex Periodic Orbits}
\textit{\textbf{(Complex Periodic Orbits):}} \textit{Suppose that Assumptions \ref{Periodic Gait of One Agent}-\ref{Invariance to x and y} are satisfied and we employ the local feedback laws \eqref{distributed_controllers}. Let $d\in\Real^{2}-\{0\}$ be a vector and define the augmented trajectory as
$\mathcal{O}_{d}^{a}:=\{x^{a}\,|\,x_{1}=\varphi^{\star}(t),\, x_{2}=T_{d}\left(\varphi^{\star}(t)\right),\, 0\leq{t}<T\}$.
Assume that the distributed feedback controllers $\Gamma^{i}_{v,w}(t,x_{i},\Theta,\xi_{i})$, when evaluated on the augmented orbit $\mathcal{O}^{a}_{d}$, are reduced to the nominal state feedback laws for each agent, that is, for every $i\in\{1,2\}$ and $v\in\mathcal{V}$,
\begin{equation}\label{propertu_of_distributed_controllers}
\Gamma_{v,v}^{i}\left(t,x_{i},\Theta,\xi_{i}\right)=\Gamma_{v}^{\textrm{nom}}\left(t,x_{i}\right),\, \forall x^{a}\in\mathcal{O}_{d}^{a},\,\,\forall t\geq0.
\end{equation}
Then, $\mathcal{O}_{d}^{a}$ is a periodic orbit for the complex model $\Sigma^{a}$.}
\end{theorem}

\begin{proof}
Choose an arbitrary $x_{1}\in\mathcal{O}$ and let $v$ be the corresponding continuous-time domain vertex for the agent 1. Then, $x_{2}=T_{d}(x_{1})\in{}T_{d}(\mathcal{O})$ and $w=v$ is the domain vertex for the agent 2. For these state values, $p_{2}^{e}=p_{1}^{e}+\textrm{col}(d,0)$ and $\dot{p}_{2}^{e}=\dot{p}_{1}^{e}$. Consequently, the holonomic constraints \eqref{holonomic_EE} and its first-order time derivative are satisfied. We need to show that its second order time-derivative is also met. From  \eqref{propertu_of_distributed_controllers}, $u_{1}=\Gamma_{v,v}^{1}(t,x_{1},\Theta,\xi_{1})=\Gamma_{v}^{\textrm{nom}}(t,x_{1})$. This together with Assumption \ref{Invariance to x and y} implies that $u_{2}=\Gamma_{v}^{\textrm{nom}}(t,T_{d}(x_{1}))=\Gamma_{v}^{\textrm{nom}}(t,x_{1})=u_{1}$. Hence, the Lagrange multiplier $\lambda^{e}$ in \eqref{dyn_06} and \eqref{dyn_07} can be zero and thereby, $\ddot{p}_{2}^{e}=\ddot{p}_{1}^{e}$. This renders $\mathcal{O}_{d}^{a}$ invariant under the augmented continuous-time  dynamics of the complex model. An analogous reasoning can be presented for the invariance under the augmented discrete-time dynamics. Consequently, $\mathcal{O}_{d}^{a}$ is an augmented periodic orbit for the system.
\end{proof}


\section{HZD-BASED NOMINAL CONTROLLERS}
\label{HZD-BASED NOMINAL CONTROLLERS}

This section presents the nominal controllers that generate the periodic locomotion $\mathcal{O}$ for each agent. The modification of these feedback laws to develop the distributed feedback controllers for cooperative locomotion of two agents will be presented in Section \ref{SYNTHESIS OF DISTRIBUTED CONTROLLERS BASED ON LOCAL QPS}. Since the agents are identical, we drop the subscript $i\in\{1,2\}$  to simplify the presentation. During the continuous-time domain $v\in\mathcal{V}$, we consider virtual constraints with nonholonomic (i.e., relative degree one) and holonomic (i.e., relative degree two) components. The nonholonomic component is used to regulate the speed of the robot, wheres the holonomic component is used for position control. We employ standard input-output (I-O) linearization \cite{Isidori_Book} to asymptotically impose virtual constraints.

Consider an output function for the domain $v$ as $y_{v}(t,x)$ that can be decomposed as follows:
\begin{equation}\label{virtual_constraints}
y_{v}(t,x):=\begin{bmatrix}
y_{v}^{\textrm{nh}}(t,q,\dot{q})\\
y_{v}^{\textrm{h}}(t,q)
\end{bmatrix}:=\begin{bmatrix}
s(q,\dot{q})-s^{\star}(\tau,v)\\
C_{v}\left(q - q^{\star}(\tau,v)\right)
\end{bmatrix},
\end{equation}
in which the superscripts ``nh'' and ``h'' stand for the nonholonomic and holonomic components, respectively. In \eqref{virtual_constraints}, $s(q,\dot{q})\in\Real$ represents the forward speed of a point on the robot (i.e, head of the robot), $s^{\star}(\tau,v)$ denotes the desired evolution of the speed on the orbit $\mathcal{O}$ in terms of the gait timing variable and the continuous-time domain $v\in\mathcal{V}$, $\tau$ denotes the gait timing variable (i.e., phasing variable) that is taken as the scaled time for each domain, and $q^{\star}(\tau,v)$ represents the desired evolution of the configuration variables on  $\mathcal{O}$ during the domain $v$. Finally, $C_{v}$ is an output matrix that affects the stability of the gait. In particular, our previous work has shown that the proper selection of the output matrix $C_{v}$ can change the stability behaviors of the gait \cite{Hamed_Buss_Grizzle_BMI_IJRR,Hamed_Gregg_decentralized_control_IEEE_CST,Hamed_Ma_Ames_Vision60}. Differentiating the virtual constraints \eqref{virtual_constraints} yileds
\begin{align}\label{IO_Linearization}
\begin{bmatrix}
\dot{y}_{v}^{\textrm{nh}}\\
\ddot{y}_{v}^{\textrm{h}}
\end{bmatrix}=
\underbrace{\begin{bmatrix}
\textrm{L}_{g_{v}}y_{v}^{\textrm{nh}}\\
\textrm{L}_{g_{v}}\textrm{L}_{f_{v}}y_{v}^{\textrm{h}}
\end{bmatrix}}_{=:A_{v}(t,x)}u+\underbrace{\begin{bmatrix}
\textrm{L}_{f_{v}}y_{v}^{\textrm{nh}}\\
\textrm{L}_{f_{v}}^{2}y_{v}^{v}
\end{bmatrix}}_{=:b_{v}(t,x)}=-\underbrace{\begin{bmatrix}
K_{p}\,y_{v}^{\textrm{nh}}\\
K_{p}\,y_{v}^{\textrm{h}}+K_{d}\,\dot{y}_{v}^{\textrm{h}}
\end{bmatrix}}_{=:e_{v}(t,x)},
\end{align}
where $K_{p}$ and $K_{d}$ are positive-definite gains. From \eqref{IO_Linearization}, one can solve for the nominal state feedback law as follows:
\begin{equation}
u=\Gamma^{\textrm{nom}}_{v}(t,x):=-A_{v}^\top\left(A_{v}\,A_{v}^\top\right)^{-1}\left(b_{v}+e_{v}\right)
\end{equation}
that will result in the asymptotic output tracking, i.e., $\lim_{t\rightarrow\infty}y_{v}(t)=0$. Here, we assume that the decoupling matrix $A_{v}$ is full-rank with a number of rows less than or equal to the number of actuators (i.e., columns).


\vspace{-0.3em}
\section{SYNTHESIS OF DISTRIBUTED CONTROLLERS}
\label{SYNTHESIS OF DISTRIBUTED CONTROLLERS BASED ON LOCAL QPS}

The objective of this section is to synthesize the distributed feedback controllers that satisfy the properties of Section \ref{DISTRIBUTED FEEDBACK CONTROLLERS}. We will make use of two local real-time QPs (one for each agent) to synthesize these controllers. For this purpose, we consider the following set of measurable global variables
\begin{equation*}
    \Theta\!:=\!\textrm{col}\left(s_{1},s_{2},q_{1}^{\textrm{roll}},q_{2}^{\textrm{roll}},q_{1}^{\textrm{pitch}},q_{2}^{\textrm{pitch}},
    \dot{q}_{1}^{\textrm{roll}},\dot{q}_{2}^{\textrm{roll}},\dot{q}_{1}^{\textrm{pitch}},\dot{q}_{2}^{\textrm{pitch}}\right)
\end{equation*}
that are available for both subsystems. Here, $s_{i}$, $q_{i}^{\textrm{roll}}$, $q_{i}^{\textrm{pitch}}$, $\dot{q}_{i}^{\textrm{roll}}$, and $\dot{q}_{i}^{\textrm{pitch}}$ represent the forward speed, roll angle, pitch angle, roll angular velocity, and pitch angular velocity for the agent $i\in\{1,2\}$, respectively. Suppose further that the agents $i\in\{1,2\}$ and $j\neq{}i\in\{1,2\}$ are in the continuous-time domains $v$ and $w$, respectively. We would like to modify the virtual constraints for agents to allow stable cooperative locomotion. Let us assume that $y_{v,w}^{i}(t,x_{i},\Theta,\xi_{i})$ denotes the \textit{modified virtual constraints} for the agent $i$. In our notation, the modified virtual constraints depend on 1) the time $t$, 2) the local state variables $x_{i}$, and 3) the measurable global variables $\Theta$. Furthermore, they are parameterized by a set of adjustable controller parameters, represented by $\xi_{i}$. One typical choice for the modified outputs can be as follows:
\begin{align}\label{modified_VCs}
    y_{v,w}^{i}\left(t,x_{i},\Theta,\xi_{i}\right):&=\begin{bmatrix}
    y_{v}^{\textrm{nh}}\left(t,q_{i},\dot{q}_{i}\right)\\
    y_{v}^{\textrm{h}}\left(t,q_{i}\right)
    \end{bmatrix}\nonumber\\
    &-\begin{bmatrix}
    \alpha_{v,w}\left(s_{j}-s^{\star}(\tau,w\right)\\
    C_{v,w}^{\textrm{roll}}\left(q_{j}^{\textrm{roll}}-q^{\star,\textrm{roll}}\left(\tau,w\right)\right)
    \end{bmatrix}\nonumber\\
    &-\begin{bmatrix}
    0\\
    C_{v,w}^{\textrm{pitch}}\left(q_{j}^{\textrm{pitch}}-q^{\star,\textrm{pitch}}\left(\tau,w\right)\right)
    \end{bmatrix},
\end{align}
where $\alpha_{v,w}$ as well as $C_{v,w}^{\textrm{roll}}$ and $C_{v,w}^{\textrm{pitch}}$ are scalars and vectors with proper dimensions to be determined. In \eqref{modified_VCs}, the nonholonomic output for the agent $i$ is corrected according to the additive term $-\alpha_{v,w}(s_{j}-s^{\star}(\tau,w))$ which takes into account the speed of the agent $j$. Note that the original nonholonomic term $y^{\textrm{nh}}_{v}(t,q_{i},\dot{q}_{i})$, defined in \eqref{virtual_constraints}, vanishes on the continuous-time domain $v$ of the orbit $\mathcal{O}$. According to the construction procedure, the additive term is also zero on the domain $w$ of the same orbit. Hence, on the domain $(v,w)$ of the complex gait $\mathcal{O}_{d}^{a}$, the modified nonholonomic term is zero. The additive terms for the holonomic portion of the modified output include the roll and pitch measurements of the other agent that are given by the terms $-C_{v,w}^{\textrm{roll}}(q_{j}^{\textrm{roll}}-q^{\star,\textrm{roll}}(\tau,w))$ and $-C_{v,w}^{\textrm{pitch}}(q_{j}^{\textrm{pitch}}-q^{\star,\textrm{pitch}}(\tau,w))$, respectively. In an analogous manner, one can show that the modified holonomic output is zero on the continuous-time domain $(v,w)$ of the cooperative gait $\mathcal{O}_{d}^{a}$. In Section \ref{NUMERICAL RESULTS}, we will show that the stability of the complex gait depends on the proper selection of the parameters $\alpha_{v,w}$, $C_{v,w}^{\textrm{roll}}$, and $C_{v,w}^{\textrm{pitch}}$. Let us define the adjustable controller parameters as follows:
\begin{equation}\label{controller_parameters}
\xi_{i}:=\left\{\alpha_{v,w},C_{v,w}^{\textrm{roll}},C_{v,w}^{\textrm{pitch}}\right\}_{(v,w)\in\mathcal{V}^{a}}.
\end{equation}
We remark that in general the controller parameters can be different for the agents. However, since the models of two agents are assumed to be identical, we assume that $\xi_{i}=\xi_{j}$ for $i,j\in\{1,2\}$.

We are now interested in regulating the modified outputs. Eliminating the Lagrange multipliers in \eqref{dyn_06}-\eqref{dyn_10} will result in the following coupled dynamics for the vertex $(v,w)$ of the complex graph
\begin{equation}\label{coupled_dyn}
    \begin{bmatrix}
    D_{1} & 0\\
    0 & D_{2}
    \end{bmatrix}\begin{bmatrix}
    \ddot{q}_{1}\\
    \ddot{q}_{2}
    \end{bmatrix} + \begin{bmatrix}
    H_{v,w,1}^{a}\\
    H_{v,w,2}^{a}
    \end{bmatrix}=\begin{bmatrix}
    B_{v,w,11}^{a} & B_{v,w,12}^{a}\\
    B_{v,w,21}^{a} & B_{v,w,22}^{a}
    \end{bmatrix}\begin{bmatrix}
    u_{1}\\
    u_{2}
    \end{bmatrix},
\end{equation}
where $D_{i}:=D(q_{i})$ for $i\in\{1,2\}$, and $H_{v,w,i}^{a}$ and $B_{v,w,ij}^{a}$ depend on the augmented state variables $x^{a}$ for $i,j\in\{1,2\}$. Since 1) the agent $i$ does not have access to all state variables of the agent $j$, and 2) the agent $i$ cannot make decision for the control action of the agent $j\neq{}i$ (i.e., $u_{j}$), we need to approximate the coupled dynamics \eqref{coupled_dyn} for the I-O linearization purpose of the modified output $y^{i}_{v,w}$. As the agent $i$ has access to its own local state variables $x_{i}$ as well as the measurable global variables $\Theta$, it is reasonable to assume that it can approximate the ``remaining'' portion of the augmented state variables by their desired evolution on the domain $w$ of the orbit $\mathcal{O}$ at any time $t$. Using this assumption, the agent $i$ can approximate its own coupled dynamics in  \eqref{coupled_dyn} as follows:
\begin{equation}\label{approx_dyn}
    D_{i}\,\ddot{q}_{i} + \hat{H}_{v,w,i}^{a} = \hat{B}_{v,w,ii}^{a}\,u_{i} + \hat{B}_{v,w,ij}^{a}\, u_{j}^{\star}(t,w),
\end{equation}
where $u_{j}^{\star}(t,w)$ represents the feedforward torques on domain $w$ of the orbit $\mathcal{O}$. Moreover, $\hat{H}_{v,w,i}^{a}$, $\hat{B}_{v,w,ii}^{a}$, and $\hat{B}_{v,w,ij}^{a}$ are approximations of $H_{v,w,i}^{a}$, $B_{v,w,ii}^{a}$, and $B_{v,w,ij}^{a}$, respectively, using the above-mentioned assumption.

Next, the I-O linearization along \eqref{approx_dyn} yields
\begin{alignat}{4}
&A_{v,w}^{i}\left(t,x_{i},\Theta,\xi_{i}\right) u_{i} &&+ b_{v,w}^{i}\left(t,x_{i},\Theta,\xi_{i}\right)\nonumber\\
& &&= -e_{v,w}^{i}\left(t,x_{i},\Theta,\xi_{i}\right),\label{modified_IO_Linearization}
\end{alignat}
where $A_{v,w}^{i}$, $b_{v,w}^{i}$, and $e_{v,w}^{i}$ are the extensions of the terms that were computed for the locomotion of a single agent  in \eqref{IO_Linearization}. Now we are in a position to present the local QP for each agent. The objective of the local QPs is to modify the control inputs for each agent to be close enough to its own nominal HZD-based control while imposing the modified virtual constraints for cooperative locomotion. More specifically, we consider the following local QP for the agent $i\in\{1,2\}$ at every time sample (e.g., $1$kHz) (see Fig. \ref{Unified_Illsuartion}b)
\begin{alignat}{4}
&\min_{(u_{i},\delta)} \,\,\, && \frac{1}{2}\left\|u_{i} - \Gamma_{v}^{\textrm{nom}}\left(t,x_{i}\right)\right\|^{2} + \frac{w}{2} \|\delta\|^{2}\label{QP}\\
&\textrm{s.t.} && A_{v,w}^{i} \,u_{i} + b_{v,w}^{i} + \delta = -e_{v,w}^{i}\label{eq_const}\\
& && u_{\min}\leq{u_{i}}\leq{}u_{\max},\,\,\delta_{\min}\leq\delta\leq\delta_{\max},\label{ineq_const}
\end{alignat}
in which $w>0$ is a weighting factor, and $\delta$ is a defect variable to satisfy the equality constraint \eqref{eq_const} in case \eqref{modified_IO_Linearization} cannot be met. This may happen if the modified decoupling matrix $A_{v,w}^{i}$ is not full-rank. The cost function \eqref{QP} tries to make the local controller $u_{i}$ close enough to the nominal controller while keeping the $2$-norm of the defect variable small. Inequality constraints \eqref{ineq_const} ensure the feasibility of the applied torques, where $u_{\min}$, $u_{\max}$, $\delta_{\min}$, and $\delta_{\max}$ denote the lower and upper bounds for the decision variables.

\begin{remark}
As the nominal HZD-based controllers stabilize the motion of each single, we would like to find an optimal solution that is close enough to the nominal controller while satisfying the modified virtual constraints. The solution of the QP will be then a combination of the nominal HZD-based controller and modified virtual constraint controllers, and this combination will be parameterized by the weighing factor $w$. We have observed that this combination is important in stabilizing cooperative gaits (see Section \ref{NUMERICAL RESULTS}).
\end{remark}

\begin{theorem}\label{Properties of the QP Solutions}
\textit{\textbf{(Properties of the Optimal Local Controllers):}} \textit{The solutions of the local QPs \eqref{QP}-\eqref{ineq_const} satisfy the property \eqref{propertu_of_distributed_controllers}.}
\end{theorem}

\begin{proof}
According to the construction procedure, for every $x^{a}\in\mathcal{O}_{d}^{a}$, $(v,w)\in\mathcal{V}^{a}$, $i\in\{1,2\}$, and any controller parameters $\xi_{i}$, the modified virtual constraints $y_{v,w}^{i}(t,x_{i},\Theta,\xi_{i})$ and $\frac{\textrm{d}}{\textrm{d}t}y_{v,w}^{i}(t,x_{i},\Theta,\xi_{i})$ (for holonomic portions) are zero. In addition, the approximate terms $\hat{H}_{v,w}^{a}$, $\hat{B}_{v,w,ii}^{a}$, and $\hat{B}_{v,w,ij}^{a}$ are equal to their actual values on the orbit $\mathcal{O}_{d}^{a}$. Hence, the approximate dynamics \eqref{approx_dyn} as well as the approximate I-O linearization in \eqref{modified_IO_Linearization} become the precise ones for the complex dynamics on the orbit $\mathcal{O}_{d}^{a}$. This together with the fact that the modified outputs are zero on $\mathcal{O}_{d}^{a}$ implies that the optimal solution of local QPs \eqref{QP}-\eqref{ineq_const} on $\mathcal{O}_{d}^{a}$ is equal to the nominal HZD-based controllers which completes the proof.
\end{proof}

\begin{corollary}\label{Complex Gait with Local QPs Corollary}
\textit{\textbf{(Complex Gait with Local QPs):}} \textit{Suppose that Assumptions \ref{Periodic Gait of One Agent}-\ref{Invariance to x and y} are satisfied and we employ the optimal local controllers in \eqref{QP}-\eqref{ineq_const}. Then, $\mathcal{O}_{d}^{a}$ is a periodic orbit for the complex hybrid model $\Sigma^{a}$.}
\end{corollary}

\begin{proof}
The proof is a result of Theorems \ref{Complex Periodic Orbits} and \ref{Properties of the QP Solutions}.
\end{proof}

\begin{remark}
\textit{\textbf{(Stability Modulo $d$):}} One immediate result from Corollary \ref{Complex Gait with Local QPs Corollary} and Assumption \ref{Invariance to x and y} is that the proposed local controllers can stabilize the cooperative locomotion by the proper selection of the controller parameters $\xi_{i}, i\in\{1,2\}$, but \textit{cannot} stabilize the location of the agents with respect to each other. In particular, $\mathcal{O}_{d}^{a}$ for any $d$ with the property $d\neq0$ can be a periodic orbit. We refer to this stability, \textit{stability modulo $d$}. Full-state stability requires the measurement of the absolute Cartesian positions of the agents which is \textit{not} considered in this paper.
\end{remark}


\begin{figure*}[!t]
\centering
\subfloat[\label{Original_VCs}]{\includegraphics[width=2.0in]{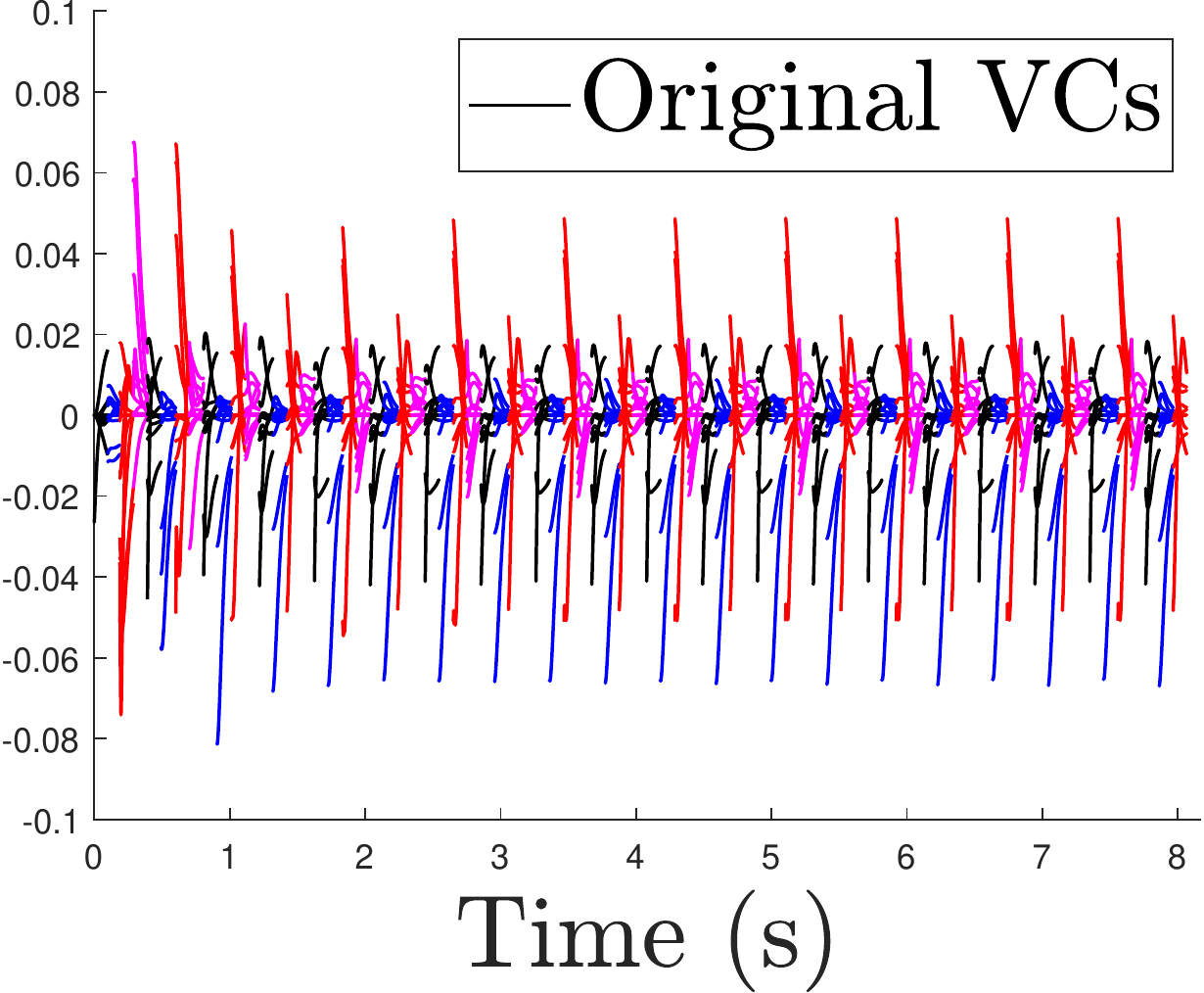}}\quad\,\,
\subfloat[\label{Original_VCs_Unstable}]{\includegraphics[width=2.3in]{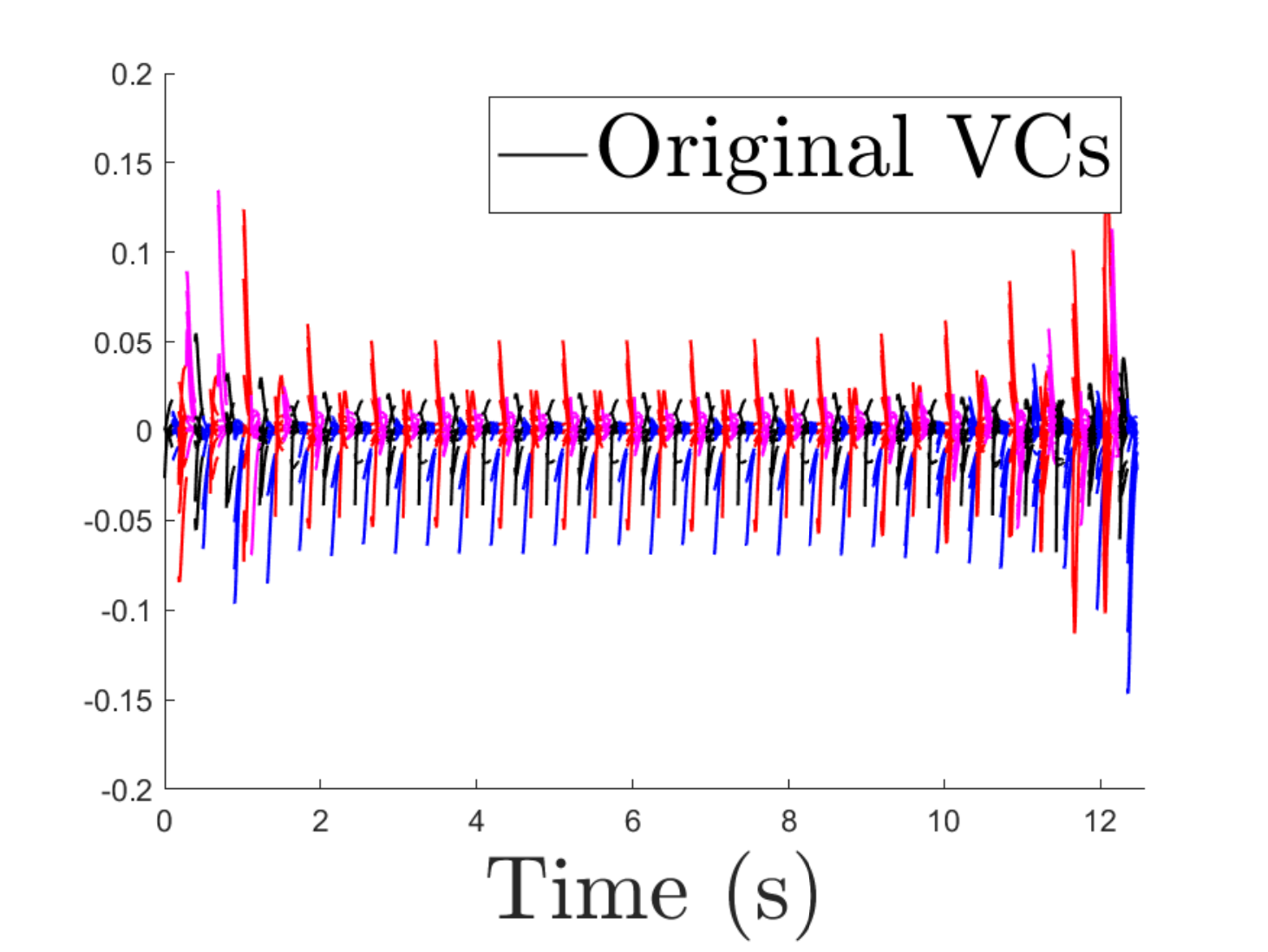}}
\subfloat[\label{Modified_VCs}]{\includegraphics[width=2.3in]{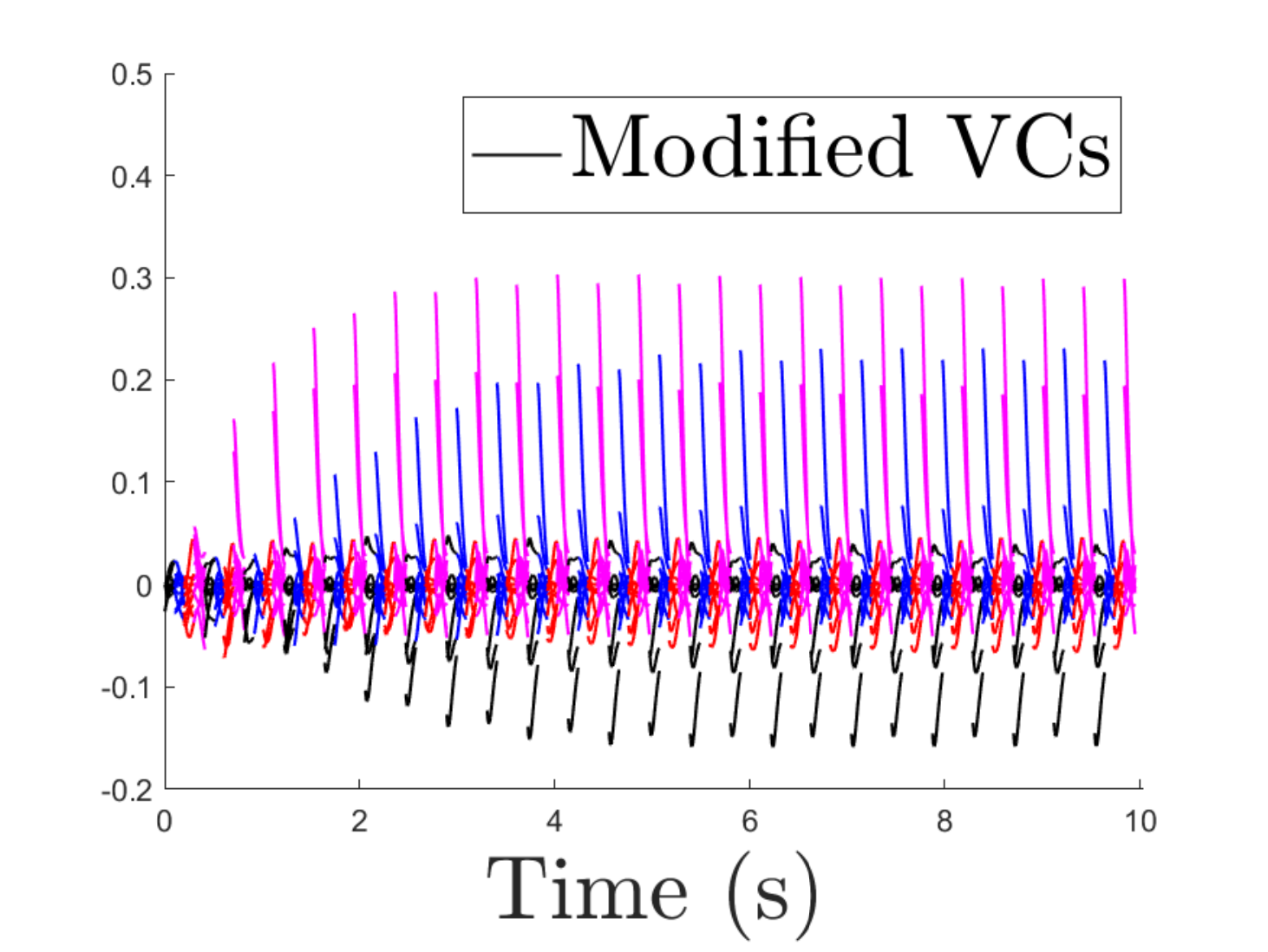}}
\vspace{-0.7em}
\caption{(a) Plot of the original virtual constraints with the nominal HZD controller for locomotion of a single agent versus time. The gait is stable. Colors distinguish different domains of locomotion. (b) Plot of the original virtual constraints for cooperative locomotion of two agents with the nominal HZD-based control (i.e., no modified virtual constraints and no QP). The complex gait is unstable. (c) Plot of the modified virtual constraints for cooperative locomotion of two agents with the proposed distributed controllers. The complex gait is stable.}
\label{l}
\vspace{-1.3em}
\end{figure*}

\begin{figure*}[!t]
\centering
\vspace{0.01em}
\subfloat[\label{phaseportrait_roll}]{\includegraphics[width=2.0in]{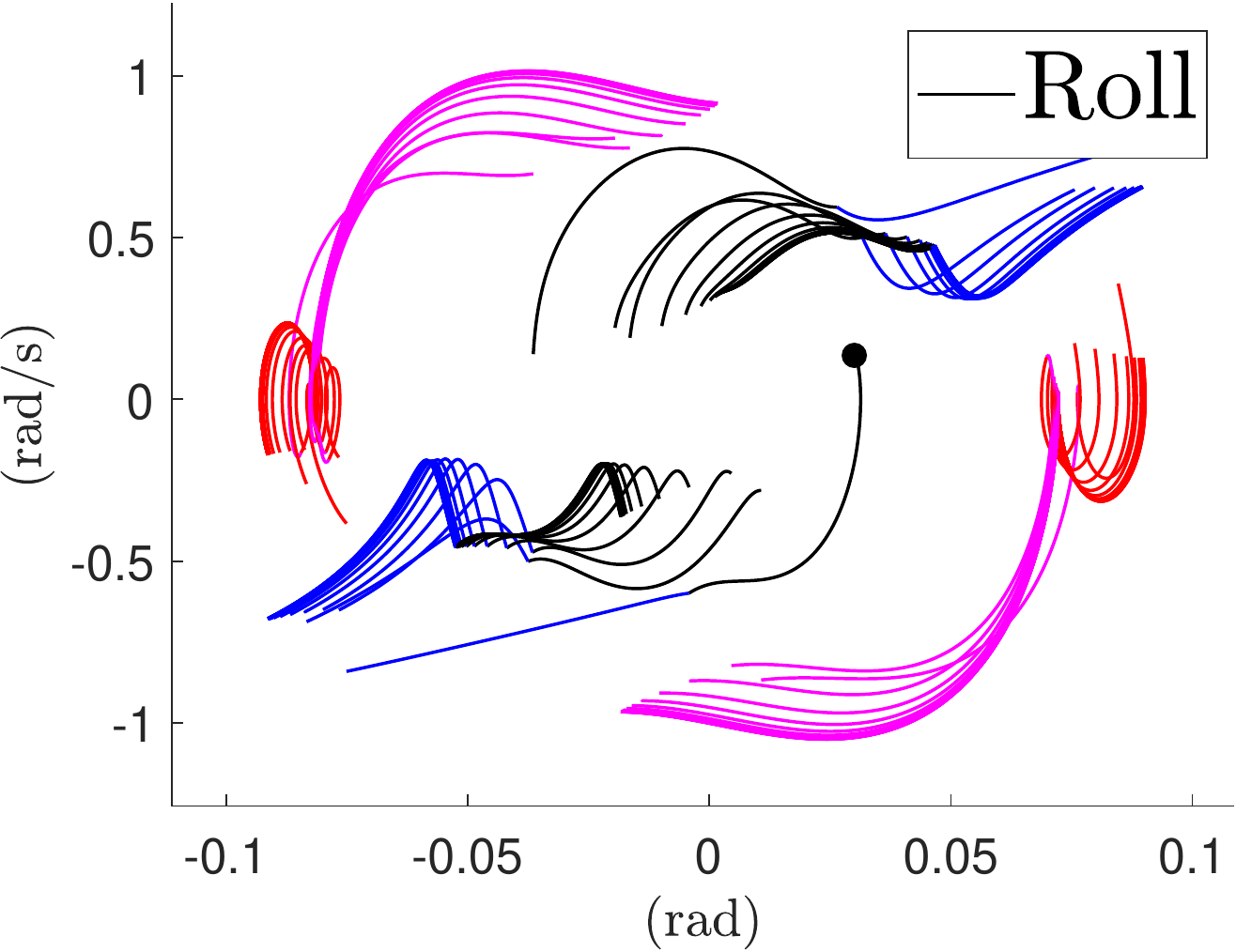}}
\subfloat[\label{phaseportrait_pitch}]{\includegraphics[width=2.0in]{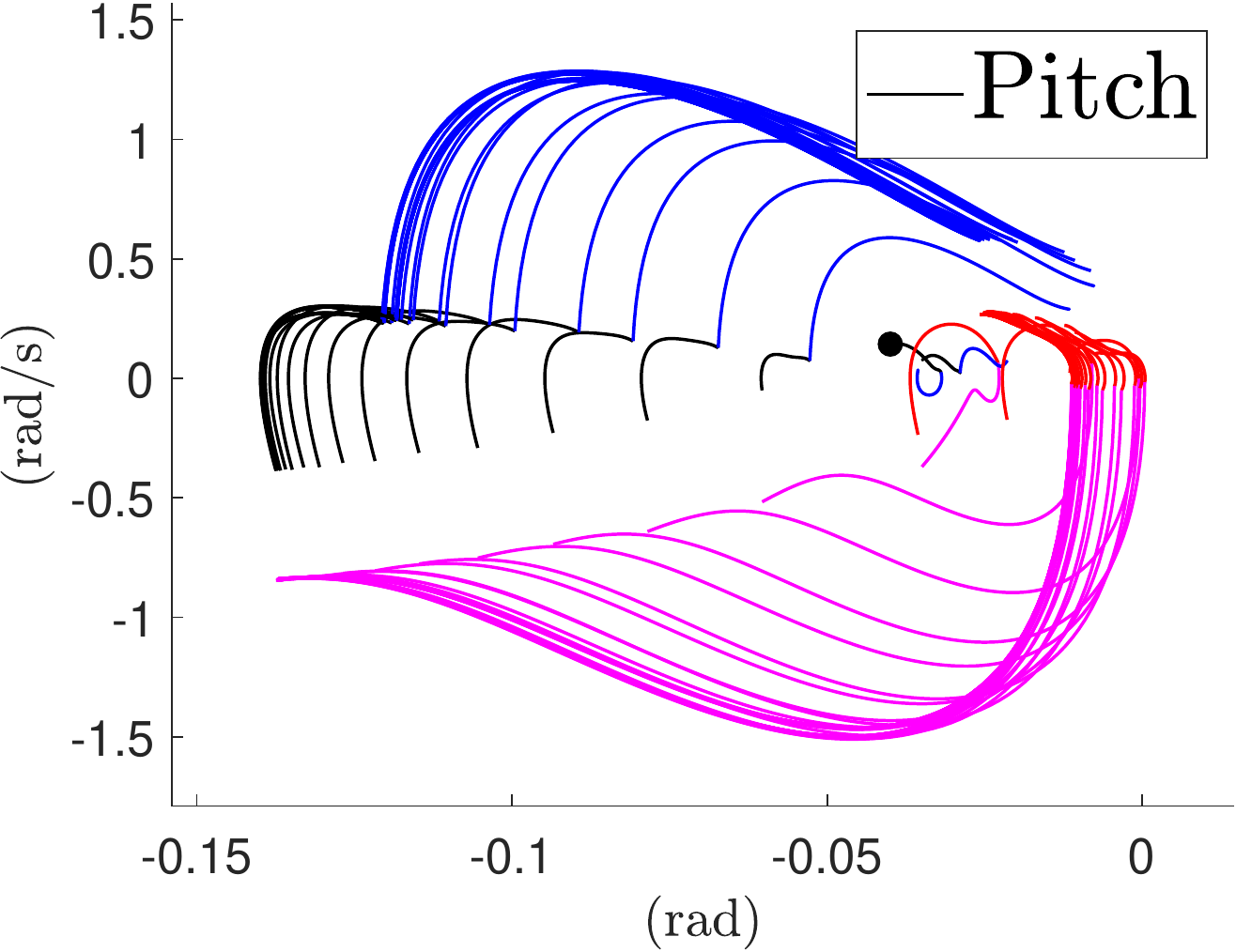}}
\subfloat[\label{phaseportrait_yaw}]{\includegraphics[width=2.0in]{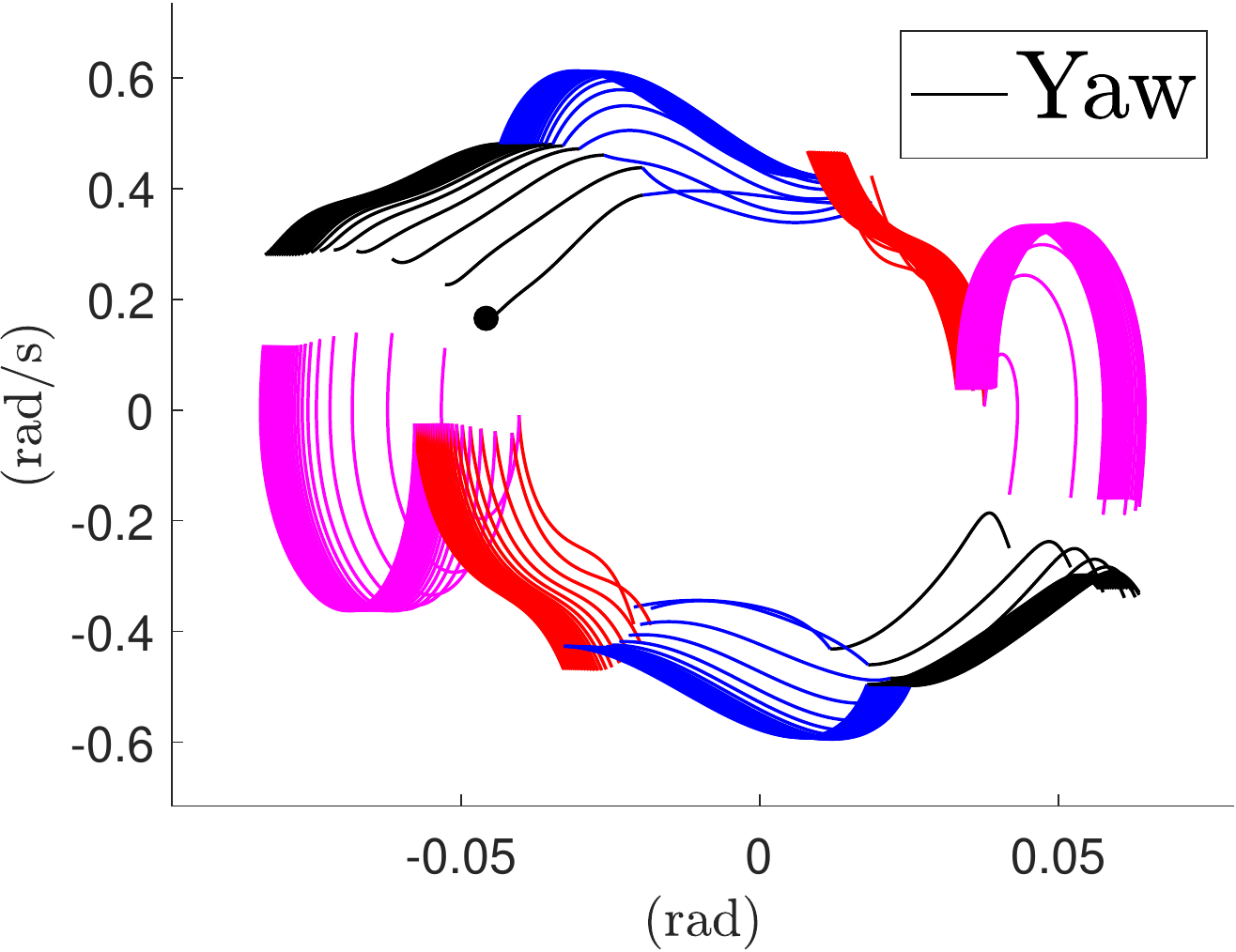}}
\vspace{-0.7em}
\caption{Phase portraits for the complex hybrid model of cooperative locomotion with the proposed distributed controllers. }
\label{phase_portraits_distributed_control}
\vspace{-1.5em}
\end{figure*}

\vspace{-0.3em}
\section{NUMERICAL RESULTS}
\label{NUMERICAL RESULTS}

The objective of this section is to numerically evaluate the effectiveness of the proposed distributed controllers for stabilizing cooperative locomotion of two Vision 60 robots augmented with 6 DOF Kinova arms. We consider a walking gait $\mathcal{O}$ for each agent with 8 continuous-time domains at the speed of $0.34$ (m/s) (see Fig. \ref{Unified_Illsuartion}a). The gait is designed via  FROST (Fast Robot Optimization and Simulation Toolkit) --- an open-source toolkit for path planning of dynamic legged locomotion \cite{FROST}. Our previous work in \cite{Hamed_Buss_Grizzle_BMI_IJRR} has shown that the stability of gaits in the HZD approach depends on the proper selection of the virtual constraints that is equivalent to the proper selection of the output matrices $C_{v}$ in \eqref{virtual_constraints}.

To exponentially stabilize the gait $\mathcal{O}$, we make use of the iterative optimization algorithm of \cite{Hamed_Ma_Ames_Vision60,Hamed_Gregg_decentralized_control_IEEE_CST} that was developed based on LMIs and BMIs to look for stabilizing values of $C_{v}$. Figure \ref{Original_VCs} depicts the time profile of the regular virtual constraints for stable locomotion of one agent. Here, we employ the nominal HZD-based controllers that were developed in Section \ref{HZD-BASED NOMINAL CONTROLLERS}. Convergence to a stable periodic motion is clear. Although the nominal HZD-based controllers can stabilize the walking gait for a single agent, they \textit{cannot} address the cooperative locomotion of two agents.

Figure \ref{Original_VCs_Unstable} illustrates the original virtual constraints for the complex hybrid model of collaborative locomotion, where each agent only employs its own nominal controller. Here, we assume that the agents carry a massless bar with the length of $1$ (m) (we suppose that $d=\textrm{col}(0,1)$). The initial augmented state is taken off of the orbit $\mathcal{O}_{d}^{a}$. Divergence from the periodic gait is clear which indicates the instability. To stabilize the gait, we modify the virtual constraints for each agent as proposed in \eqref{modified_VCs}. In this paper, we choose the distributed controller parameters based on the output matrices $C_{v}$ that are optimized using the BMI algorithm. Let $C_{v}^{\textrm{roll}}$ and $C_{v}^{\textrm{pitch}}$ denote the columns of the optimized matrix $C_{v}$ that correspond to the roll and pitch motions, respectively. We then choose
\begin{equation}\label{reduction_in_parameters}
C_{v,w}^{\textrm{roll}}=\beta_{v,w}\,C_{v}^{\textrm{roll}}\quad\textrm{and}\quad
C_{v,w}^{\textrm{pitch}}=\gamma_{v,w}\,C_{v}^{\textrm{pitch}},
\vspace{-0.4em}
\end{equation}
where $\beta_{v,w}$ and $\gamma_{v,w}$ are scalars to be determined. Equation \eqref{reduction_in_parameters} reduces the choice of the controllers parameters $\xi_{i}$ in \eqref{controller_parameters} to that of the scalars $\{\alpha_{v,w},\beta_{v,w},\gamma_{v,w}\}$. In this paper, we heuristically take $\alpha_{v,w}=\beta_{v,w}=\gamma_{v,w}=0.5$. Figure \ref{phase_portraits_distributed_control} illustrates the phase portraits for the roll, pitch, and yaw motions of one of the agents during cooperative locomotion. Convergence to the periodic orbit is clear. Figure \ref{Modified_VCs} depicts the time profile of modified virtual constraints. Snapshots of the cooperative locomotion are illustrated in Fig. \ref{snapshots}. The animation of these simulations can be found at \cite{YouTube_MultiagentACC2020}.



\begin{figure*}[t!]
\centering
\vspace{0.1em}
\includegraphics[width=\linewidth]{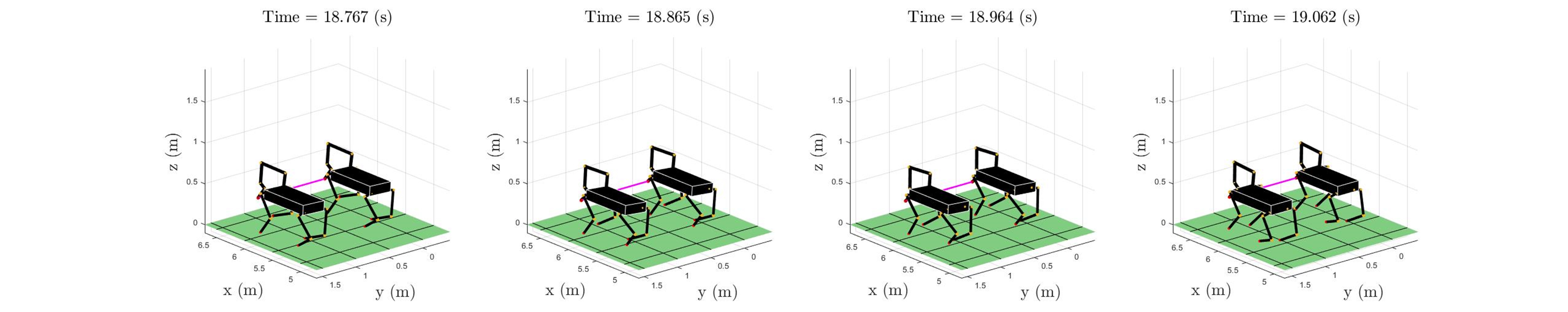}
\vspace{-1.5em}
\caption{Snapshots of four domains of the steady-state complex gait $\mathcal{O}_{d}^{a}$ stabilized by the distributed feedback controllers (see \cite{YouTube_MultiagentACC2020} for the animation of simulations.)}
\label{snapshots}
\vspace{-2em}
\end{figure*}

\vspace{-0.2em}
\section{CONCLUSIONS}
\label{CONCLUSIONS}

This paper presented an analytical approach to design distributed feedback controllers that stabilize cooperative locomotion of legged robots which are coupled to each other by holonomic constraints. We addressed complex hybrid dynamical models that represent collaborative locomotion of legged robots while steering objects with their arms. The paper studied properties and complex periodic orbits of these sophisticated and high-dimensional hybrid systems. Virtual constraints were devolved for stable locomotion of single agents and are imposed by nominal HZD-based controllers. The paper then modified the virtual constraints to achieve stable cooperative locomotion of two agents. The distributed controllers were implemented via local QPs to be close to the nominal HZD-based controllers while satisfying the modified virtual constraints. To demonstrate the power of the proposed approach, we employed the distributed controllers for an extensive numerical simulation that describes the complex hybrid model for cooperative locomotion of two Vision 60s with Kionva arms to steer an object. The complex model of locomotion has 64 continuous-time domains,  192 discrete-time transitions, 96 state variables, and 36 control inputs. For future work, we will experimentally investigate the distributed control algorithms on two Vision 60 robots with Kinova arms. We will also study safety-critical control algorithms that allow collaborative locomotion and obstacle avoidance in complex environments.


\vspace{-0.8em}
\bibliographystyle{IEEEtran}
\bibliography{references}

\begin{thebibliography}{10}
\providecommand{\url}[1]{#1}
\csname url@samestyle\endcsname
\providecommand{\newblock}{\relax}
\providecommand{\bibinfo}[2]{#2}
\providecommand{\BIBentrySTDinterwordspacing}{\spaceskip=0pt\relax}
\providecommand{\BIBentryALTinterwordstretchfactor}{4}
\providecommand{\BIBentryALTinterwordspacing}{\spaceskip=\fontdimen2\font plus
\BIBentryALTinterwordstretchfactor\fontdimen3\font minus
  \fontdimen4\font\relax}
\providecommand{\BIBforeignlanguage}[2]{{%
\expandafter\ifx\csname l@#1\endcsname\relax
\typeout{** WARNING: IEEEtran.bst: No hyphenation pattern has been}%
\typeout{** loaded for the language `#1'. Using the pattern for}%
\typeout{** the default language instead.}%
\else
\language=\csname l@#1\endcsname
\fi
#2}}
\providecommand{\BIBdecl}{\relax}
\BIBdecl

\bibitem{Murray_Control_primitives}
R.~M. Murray, D.~C. Deno, K.~S.~J. Pister, and S.~S. Sastry, ``Control
  primitives for robot systems,'' \emph{IEEE Transactions on Systems, Man, and
  Cybernetics}, vol.~22, no.~1, pp. 183--193, Jan 1992.

\bibitem{Murray_Book}
R.~Murray, Z.~Li, and S.~S.S., \emph{A Mathematical Introduction to Robotic
  Manipulation}.\hskip 1em plus 0.5em minus 0.4em\relax Taylor \& Francis/CRC,
  1994.

\bibitem{RSS2013}
K.~Sreenath and V.~Kumar, ``Dynamics, control and planning for cooperative
  manipulation of payloads suspended by cables from multiple quadrotor
  robots,'' in \emph{Robotics: Science and Systems (RSS)}, 2013.

\bibitem{MT:NM:14}
M.~Turpin, N.~Michael, and V.~Kumar, ``Capt: Concurrent assignment and planning
  of trajectories for multiple robots,'' \emph{The International Journal of
  Robotics Research}, vol.~33, no.~1, pp. 98--112, 2014.

\bibitem{DP:MT:14}
M.~T. D.~Panagou and V.~Kumar, ``Decentralized goal assignment and trajectory
  generation in multi-robot networks: A multiple lyapunov functions approach,''
  in \emph{Proc. of the 2014 IEEE Int. Conf. on Robotics and Automation}, Hong
  Kong, China, Jun. 2014.

\bibitem{Mesbahi_Book}
M.~Mesbahi and E.~M, \emph{{Graph Theoretic Methods in Multiagent
  Networks}}.\hskip 1em plus 0.5em minus 0.4em\relax Princeton University
  Press, 2010.

\bibitem{Bullo_Book}
F.~Bullo, J.~Cort\'{e}s, and M.~S, \emph{{Distributed Control of Robotic
  Networks: A Mathematical Approach to Motion Coordination Algorithms}}.\hskip
  1em plus 0.5em minus 0.4em\relax Princeton University Press, 2009.

\bibitem{Grizzle_Asymptotically_Stable_Walking_IEEE_TAC}
J.~Grizzle, G.~Abba, and F.~Plestan, ``Asymptotically stable walking for biped
  robots: {A}nalysis via systems with impulse effects,'' \emph{Automatic
  Control, IEEE Transactions on}, vol.~46, no.~1, pp. 51--64, Jan 2001.

\bibitem{Westervelt_Grizzle_Koditschek_HZD_IEEE_TRO}
E.~Westervelt, J.~Grizzle, and D.~Koditschek, ``Hybrid zero dynamics of planar
  biped walkers,'' \emph{Automatic Control, IEEE Transactions on}, vol.~48,
  no.~1, pp. 42--56, Jan 2003.

\bibitem{Chevallereau_Grizzle_3D_Biped_IEEE_TRO}
C.~Chevallereau, J.~Grizzle, and C.-L. Shih, ``Asymptotically stable walking of
  a five-link underactuated 3-{D} bipedal robot,'' \emph{Robotics, IEEE
  Transactions on}, vol.~25, no.~1, pp. 37--50, Feb 2009.

\bibitem{Ames_RES_CLF_IEEE_TAC}
A.~Ames, K.~Galloway, K.~Sreenath, and J.~Grizzle, ``Rapidly exponentially
  stabilizing control {Lyapunov} functions and hybrid zero dynamics,''
  \emph{Automatic Control, IEEE Transactions on}, vol.~59, no.~4, pp. 876--891,
  April 2014.

\bibitem{Ames_DURUS_TRO}
A.~Hereid, C.~M. Hubicki, E.~A. Cousineau, and A.~D. Ames, ``Dynamic humanoid
  locomotion: A scalable formulation for {HZD} gait optimization,'' \emph{IEEE
  Transactions on Robotics}, pp. 1--18, 2018.

\bibitem{Sreenath_Grizzle_HZD_Walking_IJRR}
K.~Sreenath, H.-W. Park, I.~Poulakakis, and J.~W. Grizzle, ``Compliant hybrid
  zero dynamics controller for achieving stable, efficient and fast bipedal
  walking on {MABEL},'' \emph{The International Journal of Robotics Research},
  vol.~30, no.~9, pp. 1170--1193, Aug. 2011.

\bibitem{Park_Grizzle_Finite_State_Machine_IEEE_TRO}
H.-W. Park, A.~Ramezani, and J.~Grizzle, ``A finite-state machine for
  accommodating unexpected large ground-height variations in bipedal robot
  walking,'' \emph{Robotics, IEEE Transactions on}, vol.~29, no.~2, pp.
  331--345, April 2013.

\bibitem{Poulakakis_Grizzle_SLIP_IEEE_TAC}
I.~Poulakakis and J.~Grizzle, ``The spring loaded inverted pendulum as the
  hybrid zero dynamics of an asymmetric hopper,'' \emph{Automatic Control, IEEE
  Transactions on}, vol.~54, no.~8, pp. 1779--1793, Aug 2009.

\bibitem{Tedrake_Robus_Limit_Cycles_CDC}
H.~Dai and R.~Tedrake, ``Optimizing robust limit cycles for legged locomotion
  on unknown terrain,'' in \emph{Decision and Control, IEEE 51st Annual
  Conference on}, Dec 2012, pp. 1207--1213.

\bibitem{Byl_HZD}
C.~O. Saglam and K.~Byl, ``Meshing hybrid zero dynamics for rough terrain
  walking,'' in \emph{2015 IEEE International Conference on Robotics and
  Automation (ICRA)}, May 2015, pp. 5718--5725.

\bibitem{Johnson_Burden_Koditschek}
A.~M. Johnson, S.~A. Burden, and D.~E. Koditschek, ``A hybrid systems model for
  simple manipulation and self-manipulation systems,'' \emph{The International
  Journal of Robotics Research}, vol.~35, no.~11, pp. 1354--1392, 2016.

\bibitem{Spong_Controlled_Symmetries_IEEE_TAC}
M.~Spong and F.~Bullo, ``Controlled symmetries and passive walking,''
  \emph{Automatic Control, IEEE Transactions on}, vol.~50, no.~7, pp.
  1025--1031, July 2005.

\bibitem{Manchester_Tedrake_LQR_IJRR}
I.~Manchester, U.~Mettin, F.~Iida, and R.~Tedrake, ``Stable dynamic walking
  over uneven terrain,'' \emph{The International Journal of Robotics Research},
  vol.~30, no.~3, pp. 265--279, 2011.

\bibitem{Ioannis_adaptation_01}
S.~{Veer}, M.~S. {Motahar}, and I.~{Poulakakis}, ``Adaptation of limit-cycle
  walkers for collaborative tasks: {A} supervisory switching control
  approach,'' in \emph{2017 IEEE/RSJ International Conference on Intelligent
  Robots and Systems (IROS)}, Sep. 2017, pp. 5840--5845.

\bibitem{Vasudevan2017}
R.~Vasudevan, \emph{Hybrid System Identification via Switched System Optimal
  Control for Bipedal Robotic Walking}.\hskip 1em plus 0.5em minus 0.4em\relax
  Cham: Springer International Publishing, 2017, pp. 635--650.

\bibitem{Hamed_Buss_Grizzle_BMI_IJRR}
K.~Akbari~Hamed, B.~Buss, and J.~Grizzle, ``Exponentially stabilizing
  continuous-time controllers for periodic orbits of hybrid systems:
  Application to bipedal locomotion with ground height variations,'' \emph{The
  International Journal of Robotics Research}, vol.~35, no.~8, pp. 977--999,
  2016.

\bibitem{Ames_HybridReduction_Original_Paper}
A.~D. Ames, R.~D. Gregg, E.~D.~B. Wendel, and S.~Sastry, ``On the geometric
  reduction of controlled three-dimensional bipedal robotic walkers,'' in
  \emph{Lagrangian and Hamiltonian Methods for Nonlinear Control 2006}.\hskip
  1em plus 0.5em minus 0.4em\relax Berlin, Heidelberg: Springer Berlin
  Heidelberg, 2007, pp. 183--196.

\bibitem{Jessy_Book}
E.~Westervelt, J.~Grizzle, C.~Chevallereau, J.~Choi, and B.~Morris,
  \emph{Feedback Control of Dynamic Bipedal Robot Locomotion}.\hskip 1em plus
  0.5em minus 0.4em\relax Taylor \& Francis/CRC, 2007.

\bibitem{Ramezani_Hurst_Hamed_Grizzle_ATRIAS_ASME}
A.~Ramezani, J.~Hurst, K.~Akbai~Hamed, and J.~Grizzle, ``Performance analysis
  and feedback control of {ATRIAS}, a three-dimensional bipedal robot,''
  \emph{Journal of Dynamic Systems, Measurement, and Control December, ASME},
  vol. 136, no.~2, December 2013.

\bibitem{Buss_Hamed_Griffin_Grizzle_BMI_ACC}
B.~Buss, K.~Akbari~Hamed, B.~A. Griffin, and J.~W. Grizzle, ``Experimental
  results for {3D} bipedal robot walking based on systematic optimization of
  virtual constraints,'' in \emph{2016 American Control Conference (ACC)}, July
  2016, pp. 4785--4792.

\bibitem{Martin_Schmiedeler_IJRR}
A.~E. Martin, D.~C. Post, and J.~P. Schmiedeler, ``{The effects of foot
  geometric properties on the gait of planar bipeds walking under HZD-based
  control},'' \emph{The International Journal of Robotics Research}, vol.~33,
  no.~12, pp. 1530--1543, 2014.

\bibitem{Hamed_Ma_Ames_Vision60}
K.~{Akbari Hamed}, W.~{Ma}, and A.~D. {Ames}, ``Dynamically stable 3d
  quadrupedal walking with multi-domain hybrid system models and virtual
  constraint controllers,'' in \emph{2019 American Control Conference (ACC)},
  July 2019, pp. 4588--4595.

\bibitem{Ioannis2016}
Q.~Cao and I.~Poulakakis, ``Quadrupedal running with a flexible torso: control
  and speed transitions with sums-of-squares verification,'' \emph{Artificial
  Life and Robotics}, vol.~21, no.~4, pp. 384--392, Dec 2016.

\bibitem{Gregg_Virtual_Constraints_Powered_Prosthetic_IEEE_TRO}
R.~Gregg, T.~Lenzi, L.~Hargrove, and J.~Sensinger, ``Virtual constraint control
  of a powered prosthetic leg: From simulation to experiments with transfemoral
  amputees,'' \emph{Robotics, IEEE Transactions on}, vol.~30, no.~6, pp.
  1455--1471, Dec 2014.

\bibitem{Grizzle_Decentralized}
A.~Agrawal, O.~Harib, A.~Hereid, S.~Finet, M.~Masselin, L.~Praly, A.~Ames,
  K.~Sreenath, and J.~Grizzle, ``First steps towards translating {HZD} control
  of bipedal robots to decentralized control of exoskeletons,'' \emph{IEEE
  Access}, vol.~5, pp. 9919--9934, 2017.

\bibitem{Hamed_Gregg_IEEE_TAC}
K.~Akbari~Hamed and R.~D. Gregg, ``Decentralized event-based controllers for
  robust stabilization of hybrid periodic orbits: Application to underactuated
  3d bipedal walking,'' \emph{IEEE Transactions on Automatic Control}, pp.
  1--16, July 2018.

\bibitem{Hamed_Gregg_decentralized_control_IEEE_CST}
------, ``Decentralized feedback controllers for robust stabilization of
  periodic orbits of hybrid systems: Application to bipedal walking,''
  \emph{Control Systems Technology, IEEE Transactions on}, vol.~25, no.~4, pp.
  1153--1167, July 2017.

\bibitem{Hamed_Kamidi_Ma_Leonessa_Ames_Dog_Human}
K.~{Akbari Hamed}, V.~R. {Kamidi}, W.~{Ma}, A.~{Leonessa}, and A.~D. {Ames},
  ``Hierarchical and safe motion control for cooperative locomotion of robotic
  guide dogs and humans: A hybrid systems approach,'' \emph{IEEE Robotics and
  Automation Letters}, pp. 1--1, 2019.

\bibitem{Hurmuzlu_Impact}
Y.~Hurmuzlu and D.~B. Marghitu, ``Rigid body collisions of planar kinematic
  chains with multiple contact points,'' \emph{The International Journal of
  Robotics Research}, vol.~13, no.~1, pp. 82--92, 1994.

\bibitem{Isidori_Book}
A.~Isidori, \emph{Nonlinear Control Systems}.\hskip 1em plus 0.5em minus
  0.4em\relax Springer; 3rd edition, 1995.

\bibitem{FROST}
A.~Hereid and A.~D. Ames, ``{FROST: Fast robot optimization and simulation
  toolkit},'' in \emph{2017 IEEE/RSJ International Conference on Intelligent
  Robots and Systems (IROS)}, Sept 2017, pp. 719--726.

\bibitem{YouTube_MultiagentACC2020}
{Distributed Feedback Controllers for Stable Cooperative Locomotion of
  Quadrupedal Robots}, \url{https://youtu.be/CTKTcLIusmA}.

\end{thebibliography}

\end{document}